\documentclass[10pt,reqno]{amsart}

\usepackage{amscd}
\usepackage{amsthm}
\usepackage{amssymb}

\usepackage{enumerate,array,mathrsfs, a4wide}
\usepackage{graphicx}
\usepackage{xcolor}

\usepackage{rotating}
\usepackage{bbm}
\usepackage[all,cmtip]{xy}

\usepackage[latin1]{inputenc}
\usepackage{dsfont}

\usepackage{hyperref}

\newtheorem{theorem}{Theorem}[section]

\newtheorem{lemma}[theorem]{Lemma}

\newtheorem{corollary}[theorem]{Corollary}

\newtheorem*{theorem*}{Theorem}

\theoremstyle{definition}
\newtheorem{definition}[theorem]{Definition}

\newtheorem{observation}[theorem]{Observation}

\theoremstyle{remark}
\newtheorem{remark}[theorem]{Remark}

\numberwithin{equation}{section}

\newcommand{\R}{{\mathbb{R}}}

\newcommand{\Z}{{\mathbb{Z}}}

\newcommand{\vrepsilon}{\epsilon}
\renewcommand{\phi}{\varphi}
\renewcommand{\theta}{\vartheta}

\newcommand{\w}{\wedge}

\DeclareMathOperator{\id}{Id}
\DeclareMathOperator{\im}{Im}

\DeclareMathOperator{\lk}{lk}
\DeclareMathOperator{\OB}{OB}

\begin{document}
\title{A Reeb flow on the three-sphere without a disk-like global surface of section}

\author[Otto van Koert]{Otto van Koert}

\address{
Department of Mathematical Sciences and Research Institute of Mathematics, Seoul National University\\
Building 27, room 402\\
San 56-1, Sillim-dong, Gwanak-gu, Seoul, South Korea\\
Postal code 08826 
}
\email{okoert@snu.ac.kr}


\keywords{Global surfaces of section, Reeb flows, Integrable systems}

\begin{abstract}
We show that there are Reeb flows on the standard, tight three-sphere that do not admit global surfaces of section with one boundary component. In particular, the Reeb flows that we construct do not admit disk-like global surfaces of section.
These Reeb flows are constructed using integrable systems, and a connected sum construction that extends the integrable system.
\end{abstract}

\maketitle
\section{Introduction}
The purpose of this note is to show the existence of a Reeb flow on the standard, tight three-sphere that does not admit a global disk-like surface of section.
The motivation for proving such a statement comes from a famous theorem of Hofer, Wysocki and Zehnder, \cite{HWZ:GSS}, which asserts that a Reeb flow on a dynamically convex, tight three-sphere admits a global disk-like surface of section.
The Reeb condition is clearly necessary since there are flows on the three-sphere without any periodic orbits; for example, Kuperberg, \cite{K}, has constructed such flows.
The work of Hofer, Wysocki and Zehnder has been generalized in many ways by different people; we mention here Hryniewicz and Salom\~ao, \cite{H:GSS,HS:GSS}.

Given a global surface of section for a Reeb flow, the resulting return map is conjugated to an area-preserving diffeomorphism. This makes disk-like global surfaces of section particularly attractive, since a lot is known about area-preserving diffeomorphisms of an open disk. We mention the results of Franks proved in~\cite{F}, which imply that an area-preserving diffeomorphism of the disk has either one periodic point or infinitely many.
The work of Hryniewicz and Salom\~ao, \cite{HS:GSS}, shows that one can weaken the condition of dynamical convexity by imposing linking conditions, so one may wonder whether a disk-like global surface of section always exists for the tight three-sphere. We show that this is not the case.
\begin{theorem}
There is a Reeb flow on $S^3$ equipped with its standard, tight contact structure, that does not admit a global surface of section with only one boundary component. In particular, this Reeb flow does not admit a disk-like global surface of section.
\end{theorem}
To prove the theorem, we construct a dynamical system using the so-called book-connected sum, a special case of the Murasugi sum.
In this book-connected sum, we take a connected sum along balls that intersect a periodic orbit, inducing the connected sum of these orbits.
We retain control on the behavior of the orbits by virtue of having many invariant sets and a planar global surface of section with four binding orbits.
The proof is completed using symmetry of the linking number and an elementary analysis of the pairs of orbits that can occur.
The question on whether there is a Reeb flow on the tight $S^3$ without any global surface of section is left open.

The book-connected sum also gives a method to construct completely integrable Reeb flows, in the sense of Arnold-Liouville, on a large class of contact three-manifolds.
The construction is very simple, and does not deserve the name ``theorem'', so we will denote it by
\begin{observation}
\label{obs:integrable}
Suppose that $W$ is an oriented surface with boundary, and assume that $\psi:W\to W$ is the time-$1$ flow of an autonomous Hamiltonian such that $\psi|_{\partial W}=\id$. Let $Y$ denote the contact open book $Y=\OB(W,\psi)$.
Then $\R \times Y$ is a completely integrable system in the sense of Arnold-Liouville\footnote{We allow the integrals to be dependent on a ``small'' subset to obtain more interesting topology. For example, there is always a ``3-atom A'' in $Y$ in the sense of Bolsinov-Fomenko's book \cite{BF}, section~3.5.}.
Furthermore, if $(W', \psi')$ is another such pair, then the contact connected sum $Y\# Y'$ is still completely integrable.
\end{observation}
The projection to the $\R$-coordinate of the symplectization and the Hamiltonian generating the symplectomorphism $\psi$ give a pair of Hamiltonians that are in involution, and the connected sum can be performed in such a way that the Hamiltonians generating $\psi$ and $\psi'$ patch together to an autonomous Hamiltonian as we will see in Lemma~\ref{lemma:connected_sum}. Explicit formulas are given in Section~\ref{sec:integrable}.
The assumption that the flow is generated by an autonomous Hamiltonian is restrictive, but it does allow $\psi$ to be a product of Dehn twists along disjoint, separating curves.

In section 2 we collect standard facts about open books, surfaces of section and linking numbers, together with a lemma involving a connected sum operation.
The proof of the theorem is in section 3, and the appendix contains details concerning connected sums of open books.

\subsection*{Acknowledgements}
I thank Pedro Salom\~ao for helpful comments.
I was supported by NRF grant NRF-2016R1C1B2007662, which was funded by the Korean Government.
 
\section{Setup}
In the following, we will assume that $Y^3$ is a $3$-manifold with a non-vanishing vector field $X$.

\begin{definition}
\label{def:gss}
A {\bf global surface of section} for a flow $\phi_t$, generated by X, is an embedded, connected surface $S \subset Y^3$ with the following properties.
\begin{description}
 \item[(i)] The boundary of the surface consists of periodic orbits of $X$ (or an invariant set in higher dimensions).
 \item[(ii)] The vector field $X$ is transverse to the interior $\mathring{S}$ of the surface; we can and will assume that $S$ is oriented such that $X$ is positively transverse to $S$. 
 \item[(iii)] For every $x \in Y^3 \setminus \partial S$ there exists $t_+>0$ and $t_-<0$ such that $\phi_{t_+}(x) \in \mathring{S}$ and $\phi_{t_-}(x) \in \mathring{S}$.
\end{description}
\end{definition}
The main point of a global surface of section is that it can be used to discretize the dynamics by defining the global return map $\tau:\mathring{S}\to \mathring{S}$ by sending $x\mapsto \phi_{t_+(x)}(x)$, where $t_+(x)>0$ is the minimal positive return time.

We will use the linking number of knots in $S^3$ as an obstruction to a global surface of section, so let us briefly collect some properties of the linking number in $Y=S^3$.
Given knots $\gamma$ and $\delta$, choose a Seifert surface $S_\gamma$ for $\gamma$. Then define the linking number $\lk(\gamma,\delta)$ as the algebraic intersection number $S_\gamma \cdot \delta$.
This is independent of the chosen Seifert surface.
An immediate corollary of this definition is the following lemma.
\begin{lemma}
\label{lemma:positive_linking}
If a periodic orbit $\delta$ of the flow $\phi_t$ bounds a global surface of section, then $\delta$ links positively with every other periodic orbit $\gamma$ of $\phi_t$, i.e.~$\lk(\delta, \gamma)>0$.
\end{lemma}

\begin{lemma}
The linking number of two knots is symmetric, i.e.~$\lk(\gamma, \delta)=\lk(\delta, \gamma)$.
\end{lemma}
See for example \cite[Corollary 4.5.3]{GS}.

\subsection{Giroux model for contact open books and global surfaces of section}
Global surfaces of section are closely related to open books, and Giroux' original construction of contact open books is particularly well-suited to our needs, because it gives a flow that is easy to understand.
We assume that we are given a Liouville domain $(W,d\lambda)$ and a symplectomorphism $\psi$ satisfying
\begin{enumerate}
\item $\psi^*\lambda=\lambda-dU$, where $U$ is negative.
\item in a neighborhood of the boundary $\psi$ is the identity, so $\psi|_{\nu_W(\partial W)}=\id$.
\end{enumerate}
\begin{remark}
Given any compactly supported symplectomorphism, meaning condition (2) is satisfied, one can always deform it to a symplectomorphism satisfying condition (1). See for example \cite[Lemma 2.13]{vK}.

In our application, condition (1) is automatically satisfied.
Indeed, we assume that $\psi$ is the time-$1$ flow of an autonomous Hamiltonian, so $\psi =Fl^{X_H}_1$.
We note that
\[
\frac{d}{dt} {Fl^{X_H}_t}^* \lambda ={Fl^{X_H}_t}^* {\mathcal L}_{X_H} \lambda = {Fl^{X_H}_t}^*(\iota_{X_H} d\lambda + d \iota_{X_H} \lambda)
={Fl^{X_H}_t}^*(-dH +d \iota_{X_H} \lambda).
\]
This is exact and $W$ is compact, so by integration in $t$ we can choose a primitive $U$ that is negative on $W$.
\end{remark}

Given the above setup, define the mapping torus
$$
M(W,\psi):=W\times \R /(x,\phi)\sim (\psi(x),\phi+U(x)\,)
$$
The form $\alpha_M:=d\phi +\lambda$ descends to a well-defined contact form.
We will consider the set $B(W):=\partial W\times D^2$, which we will refer to as a neighborhood of the binding.
The binding itself is then the set $\partial W \times \{ 0 \}$.
We will denote the projection onto the first factor by $p_\partial:B(W)\to \partial W$.
Equip the neighborhood of the binding with the contact form
\begin{equation}
\label{eq:contact_form_binding}
\alpha_B:=h_1(r)\lambda|_{\partial W}+h_2(r)d\phi
.
\end{equation}
Here $h_1$ and $h_2$ are profile functions from $[0,1)$ to $\R$ whose behavior is
indicated in Figure~\ref{fig:functions_binding}; near $r=0$, $h_1\sim 1-r^2$, and $h_2 \sim r^2$. 
For larger $r$, the function $h_1$ is assumed to have exponential drop-off, and $h_2$ should be constant.
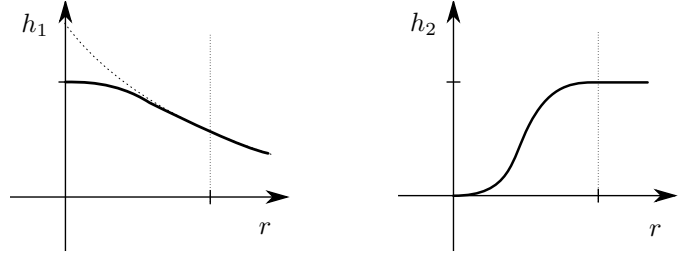
\begin{figure}[htp]
\def\svgwidth{0.6\textwidth}%
\begingroup\endlinechar=-1
\resizebox{0.6\textwidth}{!}{%
\begingroup%
  \makeatletter%
  \providecommand\color[2][]{%
    \errmessage{(Inkscape) Color is used for the text in Inkscape, but the package 'color.sty' is not loaded}%
    \renewcommand\color[2][]{}%
  }%
  \providecommand\transparent[1]{%
    \errmessage{(Inkscape) Transparency is used (non-zero) for the text in Inkscape, but the package 'transparent.sty' is not loaded}%
    \renewcommand\transparent[1]{}%
  }%
  \providecommand\rotatebox[2]{#2}%
  \newcommand*\fsize{\dimexpr\f@size pt\relax}%
  \newcommand*\lineheight[1]{\fontsize{\fsize}{#1\fsize}\selectfont}%
  \ifx\svgwidth\undefined%
    \setlength{\unitlength}{364.47853088bp}%
    \ifx\svgscale\undefined%
      \relax%
    \else%
      \setlength{\unitlength}{\unitlength * \real{\svgscale}}%
    \fi%
  \else%
    \setlength{\unitlength}{\svgwidth}%
  \fi%
  \global\let\svgwidth\undefined%
  \global\let\svgscale\undefined%
  \makeatother%
  \begin{picture}(1,0.37244992)%
    \lineheight{1}%
    \setlength\tabcolsep{0pt}%
    \put(0,0){\includegraphics[width=\unitlength,page=1]{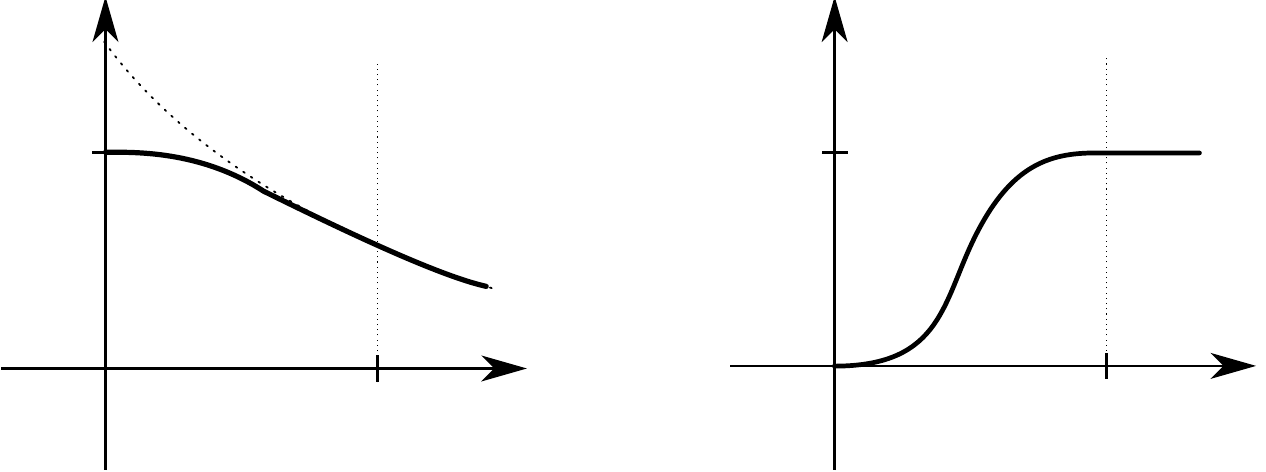}}%
    \put(0.01815833,0.32578182){\color[rgb]{0,0,0}\makebox(0,0)[lt]{\lineheight{0}\smash{\begin{tabular}[t]{l}$h_1$\end{tabular}}}}%
    \put(0.37064353,0.02670343){\color[rgb]{0,0,0}\makebox(0,0)[lt]{\lineheight{0}\smash{\begin{tabular}[t]{l}$r$\end{tabular}}}}%
    \put(0.94982294,0.02350451){\color[rgb]{0,0,0}\makebox(0,0)[lt]{\lineheight{0}\smash{\begin{tabular}[t]{l}$r$\end{tabular}}}}%
    \put(0.59532611,0.32781846){\color[rgb]{0,0,0}\makebox(0,0)[lt]{\lineheight{0}\smash{\begin{tabular}[t]{l}$h_2$\end{tabular}}}}%
  \end{picture}%
\endgroup%
}\endgroup
\caption{Functions for the contact form near the
    binding}\label{fig:functions_binding}
\end{figure}
We glue the contact manifolds together to obtain a contact open book
$$
\OB(W,\psi)= B(W) \coprod M(W,\psi) / \sim,
$$
where we identify collar neighborhoods of $B(W)$ and $M(W,\psi)$ using the map
\[
B(W)\setminus \partial W\times D^2(1/2) \longrightarrow M(W,\psi),
\quad (x; r,\phi ) \longrightarrow [Fl^{X}_{-r}(x), \phi]
\]
Here $X$ is the Liouville vector field on $W$. The above choice of $h_1$ and $h_2$ is used to show that the contact forms glue together to give a globally defined contact form.

On the complement of the binding we have a fiber bundle over the circle, given by
\[
\begin{split}
\theta:Y \setminus \partial W\times\{0\} &\longrightarrow S^1=\R/\Z \\
y & \longmapsto
\begin{cases}
\frac{\phi}{2 \pi} & \text{if }y=(w;r,\phi) \in \partial W\times D^2 \\
\frac{\phi}{-U} & y=[w,\phi] \in M(W,\psi)
\end{cases}
\end{split}
\]

The Reeb vector fields on the two models are given by
\begin{equation}
\label{eq:Reeb_vf}
\begin{split}
R_M&=\frac{\partial}{\partial \phi}
\text{ on } M(W,\psi)
\text{, and } \\
R_B&=\frac{1}{h_1(r)h_2'(r)-h_2(r)h_1'(r)}\left(
h_2'(r)R_{\lambda|_{\partial W}}
-h_1'(r)\frac{\partial}{\partial \phi}
\right)
\text{ on }
B(W).
\end{split}
\end{equation}

We conclude
\begin{lemma}
The binding is an invariant set under the Reeb flow, and each fiber of $\theta$ is a global hypersurface of section.
Furthermore, the set $B(W)$ is a disjoint union of $B\times \{ 0 \}$  and sets of the form $B\times S^1$, each of which are invariant under the Reeb flow.
\end{lemma}

\subsection{Book-connected sum of contact open books}
\label{sec:abstract_book_sum}
We describe the abstract model for the book-connected sum in the contact setting. For a description in the smooth setting, see for example \cite{MNMS}, section 7 and 8.
Suppose that $\OB(W_i,\psi_i)$ are abstract contact open books for $i=1,2$.

Denote the $j$-th boundary component of the $i$-th page $W_i$ by $b_i^j$: this labeling is a choice that will play a role later on.
Take points $p_i$ in the first boundary component of each page, so $p_i\in b_i^1$.
Define the boundary connected sum $W_1\natural_{b_1^1,b_2^1} W_2$ by attaching a Weinstein $1$-handle to $W_1\coprod W_2$ along Darboux balls in $\partial W_1\coprod \partial W_2$ containing $p_1$ and $p_2$, respectively.
This is illustrated in Figure~\ref{fig:page_sum}.
Below we will write $W_1\natural W_2:=W_1\natural_{b_1^1,b_2^1} W_2$ to remove some clutter from the notation, but it is important to realize the dependence on the choice of the additional data, and we reintroduce this notational dependence when we need this dependence.
\begin{figure}[!htb]
\def\svgwidth{1.0\textwidth}%
\begingroup\endlinechar=-1
\resizebox{1.0\textwidth}{!}{%
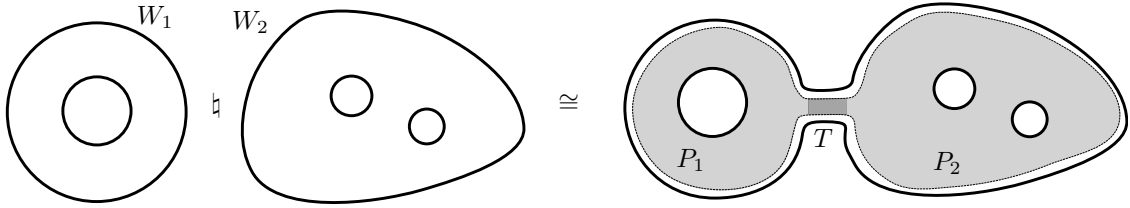%
}\endgroup
\caption{Boundary connected sum of pages}
\label{fig:page_sum}
\end{figure}
To be more precise about the attaching locus of the $1$-handle, denote the attaching locus in $W_i$ by $A_i$, and choose $A_i$ so small such that $\psi_i|_{A_i}=\id$.
We obtain embeddings
\[
\iota_j:W_j\longrightarrow W_1\natural W_2,
\]
which we use to extend the symplectomorphisms $\psi_1$ and $\psi_2$ to a symplectomorphism $\psi_1 \natural \psi_2$ of $W_1\natural W_2$.
Put
\[
\psi_1 \natural \psi_2=\begin{cases}
\iota_1\circ \psi_1(\iota_1^{-1}(x) ) & \text{if }x\in \im i_1 \\
\iota_2\circ \psi_2(\iota_2^{-1}(x) ) & \text{if }x\in \im i_2 \\
\id &\text{otherwise.}
\end{cases}
\]
The $1$-handle attachment induces a connected sum on the first binding component of $W_1$ and $W_2$.
We will denote the resulting binding component of $W_1\natural W_2$ by $b=b_1^1\#b_2^1$.
Choose a collar neighborhood $\nu_{W_1\natural W_2}(b)$ such that $\psi_1 \natural \psi_2$ restricts to the identity on this neighborhood.
Points in $W_1\natural W_2 \setminus \nu_{W_1\natural W_2}(b)$ that do not lie in the image of $\iota_1$ or $\iota_2$ will be called tube points.

This leads to the following distinguished sets:
\begin{enumerate}
\item A neighborhood of the connecting orbit,
$$
N(b)=
N(b_1^1\# b_2^1)=\{ x \in B(W_1\natural W_2) 
~|~p_\partial(x) \in b_1 \# b_2.
\}
.
$$
\item Tube orbits,
$$
T=\{ [x,\phi]\in M(W_1\natural W_2,\psi_1 \natural \psi_2)~|~x \text{ is a tube point} \}
.
$$
\item Page orbits, consisting of points that lie in the mapping torus part, minus the tube, or in a neighborhood of a binding orbit that is not $b=b_1^1\# b_2^1$. In a formula,
$$
P_j=\{ [x,\phi]\in M(W_1\natural W_2,\psi_1 \natural \psi_2)\}~|~x\in \im \iota_j \} \cup \bigcup_{b'\neq b, b'\subset Y_j} N(b').
$$
\end{enumerate}
These sets are invariant under the Reeb flow as we shall explain in Lemma~\ref{lemma:connected_sum}.
We will refer to $\OB(W_1\natural W_2,\psi_1 \natural \psi_2)$ as the abstract {\bf book-connected sum}. 
This model for the connected sum also has concrete model, meaning one for which open books are defined via a binding and a fiber bundle over the complement of the binding. 
In the concrete model, the connected sum is more obvious; simply take small balls centered at two points in a binding component and perform the usual connected sum. A local model shows that the open book extends. We won't make use of this description, though, and instead refer to the previously mentioned \cite{MNMS}.

\begin{lemma}
\label{lemma:connected_sum}
The contact open book $\OB(W_1\natural W_2,\psi_1 \natural \psi_2)$, equipped with the contact form coming from the abstract book-connected sum, is contactomorphic to the contact connected sum of the open books, $\OB(W_1,\psi_1)\# \OB(W_2,\psi_2)$.
Furthermore, for the Reeb flow of the abstract book-connected sum we have the following invariant sets
\begin{enumerate}
\item a neighborhood of the connecting orbit $b=b_1^1 \# b_2^1$, $N(b)$, as defined in Section~\ref{sec:abstract_book_sum}.
\item the set $T$ consisting of tube orbits.
\item the sets $P_1$ and $P_2$.
\end{enumerate}
If $\gamma$ is an orbit in $P_i$, then there is a Seifert surface $S_\gamma$ with $S_\gamma\subset P_i\cup N(b)$.
This Seifert surface can be chosen to intersect $N(b)$ in finitely many disks whose image under $p_\partial$ consist of finitely many points.
\end{lemma}
For intuition, the last statement says that $S_\gamma$ intersects $N(b)$ in ``flat'' disks.
\begin{proof}
The first assertion is a standard fact, which we reprove in the appendix. One ingredient of the construction in the appendix is that the connected sum (or equivalently, the region of $1$-handle attachment) is performed along Darboux balls that are contained in $N(b_1^1)$ and in $N(b_2^1)$, respectively. We will use this ingredient for the statement about the Seifert surfaces.

To verify the statement on invariant sets, we use the explicit model from equation~\eqref{eq:Reeb_vf}.
\begin{enumerate}
\item near the binding, the Reeb vector field is given by $R_B$; orbits have constant $r$-coordinate, and form therefore an invariant set.
Furthermore, these sets, which are all diffeomorphic to $\partial (W_1\natural W_2) \times S^1$, foliate $N(b)$.
\item on the set $T$, the monodromy $\psi_1 \natural \psi_2$ is defined to be the identity. The explicit form of $M(W_1\natural W_2,\psi_1\natural \psi_2)$ shows that the return map on $T$ is also the identity, so $T$ is an invariant set.
\item it follows that the complement of $N(b)$ and $T$ is also an invariant set, and this is the union $P_1\cup P_2$. 
This is a disjoint union as $P_1 \subset Y_1$ and $P_2 \subset Y_2$.
\end{enumerate}

Let us now consider the Seifert surfaces.
Take a periodic Reeb orbit $\gamma$ in $P_1$.
The latter set is disjoint from the Darboux balls used for the connected sum, so we may view $P_1\subset Y_1$.
Hence we regard $\gamma$ as an orbit in $Y_1$.
Choose a Seifert surface $S'_\gamma\subset Y_1$ capping $\gamma$. We isotope this surface such that it intersects the set $N(b_1)$ in the desired way, so that $p_\partial( S'_\gamma \cap N(b_1) )$ consists of only finitely many points, which can be done with a standard transversality argument involving a surface and a link.
Furthermore, we can arrange that these points do not lie in the Darboux ball $\iota_1(D^{3})$, yielding a Seifert surface $S_\gamma\subset Y$ satisfying the claim.
\end{proof}

\begin{remark}
We point out that the abstract contact open book changes the contact form in a neighborhood of all orbits that intersect the balls used for the connected sum because of Equation~\eqref{eq:contact_form_binding}.
Also, unlike the Weinstein model for the connected sum, the Reeb flow on the abstract book-connected sum does not have a non-degenerate Lyapunov orbit, but a degenerate family instead.
One can deform the contact form on the abstract contact open book or the return map as explained in Remark~\ref{rem:non-deg_lyapunov} to obtain a non-degenerate Lyapunov orbit in the separating sphere of the connected sum.
\end{remark}

\section{Proof of the theorem}
Lemma~\ref{lemma:connected_sum} implies the following corollary which will be used to prove the theorem.
\begin{corollary}
\label{cor:simple_linking}
Consider the setup of Section~\ref{sec:abstract_book_sum} with the same notation.
Suppose $\gamma_1$ and $\gamma_2$ are orbits in $P_1$ and $P_2$, respectively. Then
$$
\lk(\gamma_1,\gamma_2)=0.
$$
Furthermore, if $\gamma$ is an orbit in $P_1$ or in $P_2$ and $\tau$ is an orbit in $T$, then $\lk(\gamma,\tau)=0$.
\end{corollary}
Put differently, neither the $P$-orbits nor the $T$-orbits can be the boundary of a global surface which has only one boundary component.
\begin{proof}
To see that the corollary holds, assume that $\gamma\subset P_1$. Lemma~\ref{lemma:connected_sum} gives a Seifert surface $S_\gamma$ for $\gamma$ with $S_\gamma\subset P_1\cup N(b)$. It follows that $\lk(\gamma,\gamma_2)=S_\gamma\cdot \gamma_2=0$ for any periodic orbit $\gamma_2$ that is contained in $P_2$ or in $T$, since $S_\gamma$ is disjoint from these sets.
\end{proof}

We take now three open books, each representing the standard tight $S^3$, namely $Y_i=\OB(W_i,\psi_i)$, where each page is a cylinder, so $W_i\cong (I\times S^1, rd\phi)$, and each return map $\psi_i$ is a positive Dehn twist, i.e.~a map of the form $I\times S^1 \to I\times S^1,~(r,\phi) \mapsto (r,\phi+tw(r) )$, where $tw$ is a smooth function that decreases from $2\pi$ at the lower boundary to $0$ at the upper boundary.\footnote{A positive Dehn twist is often also called a right-handed Dehn twist. However, with the conventions in this paper, $r d\phi$ as Liouville form, the positive twist actually turns to the left.
Note that the argument here only involves the invariant sets and not the direction of the twists. However, negative Dehn twists will yield an overtwisted three-sphere.
}
We will denote the boundary components of the $i$-th cylinder $W_i$ by its upper component $u_i=\{1\}\times S^1$ and lower component $\ell_i=\{ -1 \}\times S^1$, respectively.

Form the book-connected sum
$$
Y=\OB(  (W_1 \natural_{u_1,u_2} W_2)\natural_{\ell_2,u_3} W_3, \psi_1\natural \psi_2 \natural \psi_3)
\cong 
\OB(  W_1 \natural_{u_1,u_2} (W_2 \natural_{\ell_2,u_3} W_3), \psi_1\natural \psi_2 \natural \psi_3)
.
$$
The page of the new open book is illustrated in Figure~\ref{fig:four_bindings}.
\begin{figure}[htp]
\def\svgwidth{0.5\textwidth}%
\begingroup\endlinechar=-1
\resizebox{0.5\textwidth}{!}{%
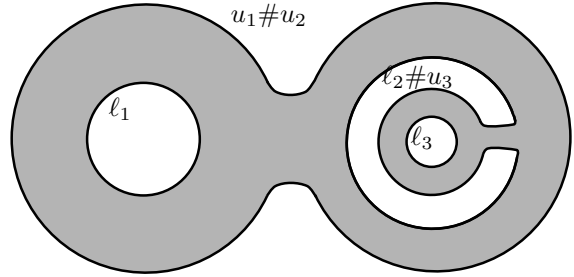%
}\endgroup
\caption{A page of an open book with four binding components}\label{fig:four_bindings}
\end{figure}
By the first statement of Lemma~\ref{lemma:connected_sum}, we see that $Y$ is contactomorphic to $(S^3,\xi_0)$ as $(S^3,\xi_0) \# (S^3,\xi_0)\cong (S^3,\xi_0)$.
We endow $Y$ with the Reeb flow from the construction from Section~\ref{sec:abstract_book_sum} and will argue by contradiction. Assume that $S$ is a global surface of section for the Reeb flow with only one boundary component, which we call $\delta$.

By Lemma~\ref{lemma:connected_sum} the contact manifold $Y$ is the disjoint union of invariant sets in the form of $P$-orbits, tube orbits $T$ and the neighborhoods of the two connecting orbits, $N(u_1 \# u_2)$ and $N(\ell_2 \# u_3)$.
Obviously, $\delta$ must be contained in one of these invariant sets.
Consider the set $P_1$, the page orbits in $W_1$ corresponding to the first connected sum. This set contains $\ell_1$, so it is not empty.
By Corollary~\ref{cor:simple_linking} applied to the first connected sum in $Y_1 \# (Y_2 \# Y_3)$ we see that $\lk(\ell_1, \delta)=0$ unless
\begin{enumerate}
\item $\delta \subset P_1$, or
\item $\delta \subset N(u_1 \# u_2)$.
\end{enumerate}
See Figure~\ref{fig:four_bindings_decomp}.
\begin{figure}[htp]
\def\svgwidth{0.5\textwidth}%
\begingroup\endlinechar=-1
\resizebox{0.5\textwidth}{!}{%
\begingroup%
  \makeatletter%
  \providecommand\color[2][]{%
    \errmessage{(Inkscape) Color is used for the text in Inkscape, but the package 'color.sty' is not loaded}%
    \renewcommand\color[2][]{}%
  }%
  \providecommand\transparent[1]{%
    \errmessage{(Inkscape) Transparency is used (non-zero) for the text in Inkscape, but the package 'transparent.sty' is not loaded}%
    \renewcommand\transparent[1]{}%
  }%
  \providecommand\rotatebox[2]{#2}%
  \newcommand*\fsize{\dimexpr\f@size pt\relax}%
  \newcommand*\lineheight[1]{\fontsize{\fsize}{#1\fsize}\selectfont}%
  \ifx\svgwidth\undefined%
    \setlength{\unitlength}{709.51666404bp}%
    \ifx\svgscale\undefined%
      \relax%
    \else%
      \setlength{\unitlength}{\unitlength * \real{\svgscale}}%
    \fi%
  \else%
    \setlength{\unitlength}{\svgwidth}%
  \fi%
  \global\let\svgwidth\undefined%
  \global\let\svgscale\undefined%
  \makeatother%
  \begin{picture}(1,0.48498303)%
    \lineheight{1}%
    \setlength\tabcolsep{0pt}%
    \put(0,0){\includegraphics[width=\unitlength,page=1]{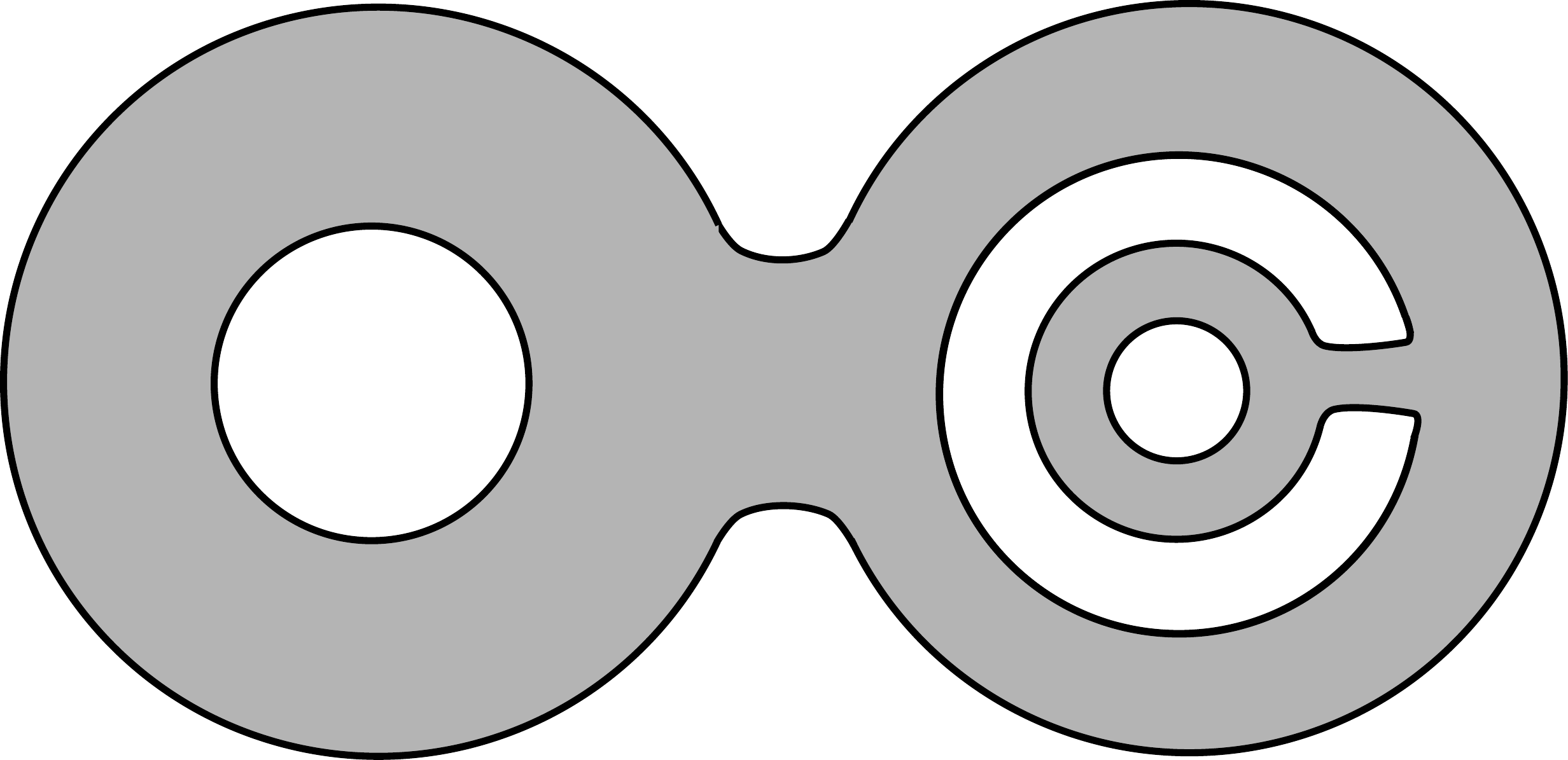}}%
    \put(0.37799748,0.44768856){\color[rgb]{0,0,0}\makebox(0,0)[lt]{\lineheight{0}\smash{\begin{tabular}[t]{l}$N(u_1\#u_2)$\end{tabular}}}}%
    \put(0.1978258,0.37529243){\color[rgb]{0,0,0}\makebox(0,0)[lt]{\lineheight{0}\smash{\begin{tabular}[t]{l}$P_1$\end{tabular}}}}%
    \put(0,0){\includegraphics[width=\unitlength,page=2]{four_bindings_decomp.pdf}}%
    \put(0.46342064,0.22316801){\color[rgb]{0,0,0}\makebox(0,0)[lt]{\lineheight{1.25}\smash{\begin{tabular}[t]{l}$T$\end{tabular}}}}%
    \put(0,0){\includegraphics[width=\unitlength,page=3]{four_bindings_decomp.pdf}}%
    \put(0.84548637,0.19665438){\color[rgb]{0,0,0}\makebox(0,0)[lt]{\lineheight{0}\smash{\begin{tabular}[t]{l}$T_2$\end{tabular}}}}%
    \put(0,0){\includegraphics[width=\unitlength,page=4]{four_bindings_decomp.pdf}}%
  \end{picture}%
\endgroup%
}\endgroup
\caption{The page and tube orbits}\label{fig:four_bindings_decomp}
\end{figure}
First assume $\delta \subset P_1$. Applying again Corollary~\ref{cor:simple_linking} we see that $\lk(\delta, \gamma_T)=0$ for any tube orbit $\gamma_T$ contradicting Lemma~\ref{lemma:positive_linking}.
Hence we assume $\delta \subset N(u_1 \# u_2)$.
Now apply Lemma~\ref{lemma:connected_sum} to the second connected sum in $(Y_1 \# Y_2) \# Y_3$. We conclude that there is a Seifert surface $S_\delta$ capping $\delta$ which is disjoint from orbits in the second tube. This implies
$$
\lk(\delta, \gamma_{T_2})=S_\delta \cdot \gamma_{T_{2}}=0
$$
for every periodic orbit in the second tube (which exist since the monodromy is the identity on the tube).
As the linking number is independent of the choice of Seifert surface, we get again a contradiction to Lemma~\ref{lemma:positive_linking}.

\begin{remark}
For the topology, the choice of bindings along which we performed the book-connected sum is irrelevant. It does matter for the dynamics. 
For example, the above argument fails for
$$
Y=\OB(  (W_1 \natural_{u_1,u_2} W_2)\natural_{u_1\# u_2,u_3} W_3, \psi_1\natural \psi_2 \natural \psi_3),
$$
because the long ``transit'' orbit $u_1\# u_2 \#u_3$ could bound a disk that links with all orbits.
\end{remark}

\subsection{Some comments on integrable systems}
\label{sec:integrable}
Let us first give explicit formulas for the integrals mentioned in the introduction.
Suppose that $Y=\OB(W,\psi)$, where $\psi$ is the time-$1$ flow of an autonomous Hamiltonian $H:W\to \R$.
Then $\R \times Y$ has two integrals that are in involution, namely
$$
H_1: \R \times Y \longrightarrow \R,\quad (t,y) \longmapsto t, \text{ and }
$$
\[
H_2: \R \times Y \longrightarrow \R,\quad (t,y) \longmapsto 
\begin{cases}
H(w) & y=[w,\phi] \in M(W,\psi) \\
h_2(r) +H(w)-h_2(1) & y=(w;r,\phi)\in B(W).
\end{cases}
\]
We note here that the condition $\psi|_{\partial W}=\id$ implies that $H|_{\partial W}$ is locally constant. However, $H|_{\partial W}$ does not need to be constant; its values can differ for the different boundary components.
These formulas prove the first assertion of Observation~\ref{obs:integrable}. For the second assertion, we use the abstract book-connected sum to perform the connected sum.

We note that this construction only gives $C^\infty$-integrals due to the cutoff function involved in the definitions of $h_2$ and $H$.

\begin{remark}
\label{rem:non-deg_lyapunov}
By the above, we can realize the flow of the proof of the main theorem as a completely integrable system. The way we described to do so yields a degenerate system. 
In this case the slope of the Liouville tori is constant across a large family of them.
In Figure~\ref{fig:non-degenerate} we have sketched level sets of a Hamiltonian generating a return map that is isotopic to the one of the above system and that is non-degenerate in the sense of Kolmogorov. The critical points on the singular level sets correspond to the Lyapunov orbits of the connected sum.
\end{remark}

\begin{figure}[htp]
\def\svgwidth{0.5\textwidth}%
\begingroup\endlinechar=-1
\resizebox{0.5\textwidth}{!}{%
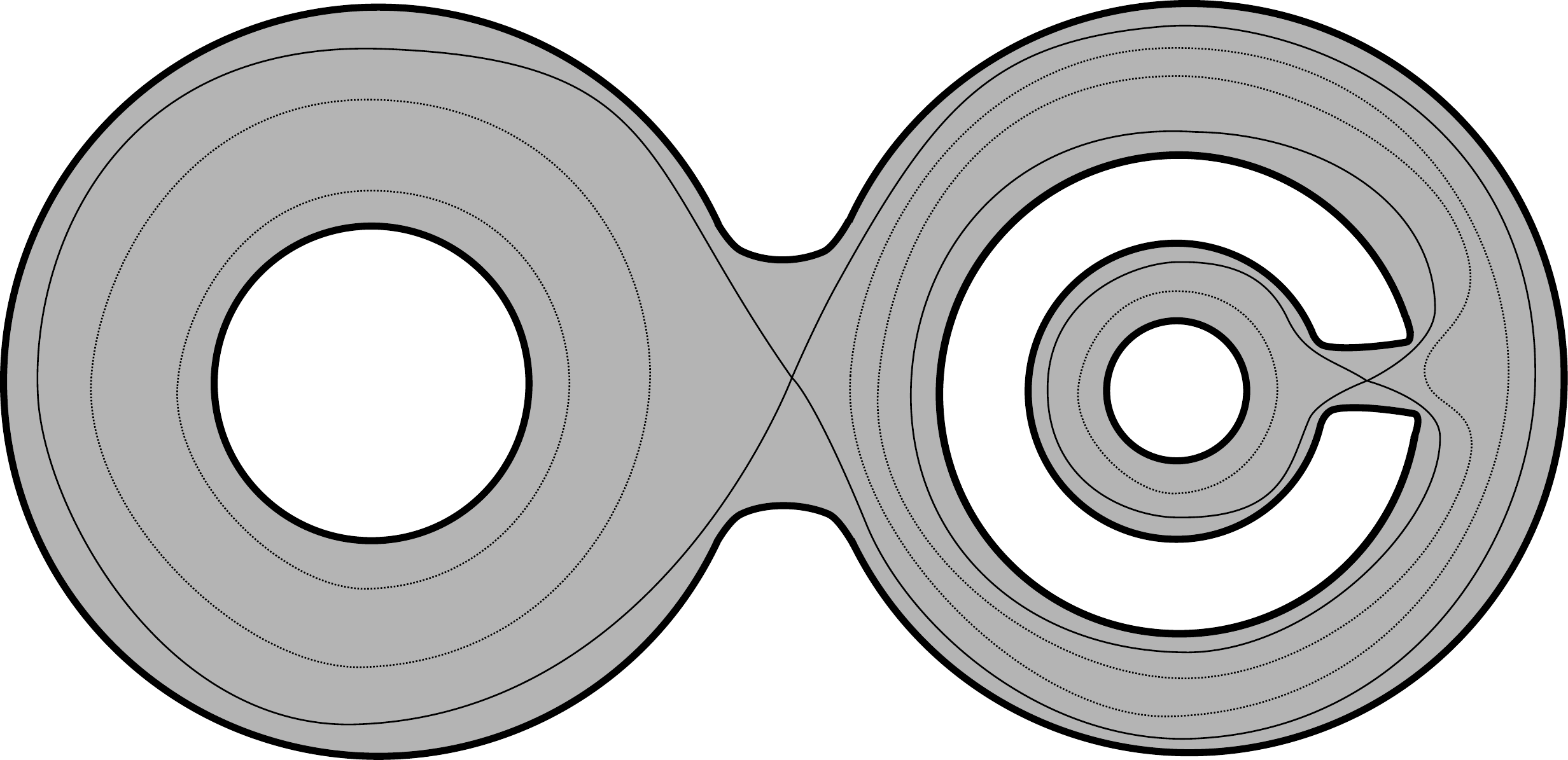%
}\endgroup
\caption{Level sets of a Hamiltonian generating $\psi_1\natural \psi_2 \natural \psi_2$}\label{fig:non-degenerate}
\end{figure}

\section{Appendix: connected sum and book-connected sum}
We will derive the first assertion of Lemma~\ref{lemma:connected_sum} from a sequence of more general statements.
Consider a Liouville cobordism $(W,\omega,X)$, so $\omega$ is a symplectic form, and $X$ is a globally defined Liouville vector field that is transverse to all boundary components.
Denote the {\bf convex boundary} of $W$ by
$$
\partial_+ W:= \{ x\in \partial W~|~ X(x) \text{ points outward} \}
.
$$
The convex boundary admits the contact form $\alpha_+:=i_X \omega |_{\partial_+ W}$.
Put $B:=\partial_+ W$, and consider an embedded isotropic sphere $S:S^{k-1} \to B$, so $S^*\alpha_+=0$.
We will impose the additional condition that the conformal symplectic normal bundle $CSN_B(S)$ is trivial.
Here is a brief, non-intrinsic description of this bundle. First of all, choose a complex structure $J$ that is compatible with $(\xi=\ker \alpha_+,d\alpha_+)$.
Denote the Reeb vector field of $\alpha_+$ by $R_+$.
Then we have
$$
T B|_{S} \cong \langle R_+ \rangle \oplus TS \oplus J TS \oplus CSN_B(S).
$$
Choose a trivialization $\vrepsilon: S\times (\R^{2n-2k},\omega_0)\to CSN_B(S)$.
As was first found by Weinstein, this data suffices for the attachment of {\bf symplectic handles}. We refer to the book of Geiges, \cite{G}, chapter 6, for a detailed description. 

Denote a $2n$-dimensional Weinstein $k$-handle by $H_k^{2n}$.
We write
$$
\tilde W|_{S,\vrepsilon}:=W \cup_{S,\vrepsilon} H^{2n}_k.
$$
We define contact surgery of $B$ along $(S,\vrepsilon)$ as the convex boundary $\partial_+ (\widetilde{W}|_{S,\vrepsilon})$. We will indicate handle attachment by $\widetilde{\phantom{~}}$ and surgery by $\overline{\phantom{~}}$, so
\begin{equation}
\label{eq:handle_attach_induces_surgery}
\partial_+ (\widetilde{W}|_{S,\vrepsilon})
=
\overline {\partial_+ W} |_{S,\vrepsilon}.
\end{equation}

\subsection{Relation with open books}
We will now describe the effect of some special subcritical contact surgeries on open books.
We start with an open book whose monodromy is trivial; for this, consider a stabilized Liouville domain $(W\times D^2,\omega_W+\omega_D,X_W+X_D)$. Here $\omega_W$ and $X_W$ stand for the symplectic form and Liouville vector field on $W$, respectively. The subscript $D$ is used for the corresponding objects on $D^2$. We will take $\omega_D=rdr\w d\phi$ and $X_D=\frac{1}{2}r\partial_r$ where $(r,\phi)$ are the usual polar coordinates.
We denote the Liouville forms by $\lambda_W=i_{X_W} \omega_W$ and $\lambda_D = i_{X_D} \omega_D$, respectively.

We obtain a natural open book on the convex boundary of $W\times D^2$ after first ``rounding''  its corners by the following procedure. 
Note that a neighborhood of the corner is given by $(-\varepsilon_W,0] \times \partial W \times (-\varepsilon_D,0] \times S^1$.
The neighborhoods $(-\varepsilon_W,0] \times \partial W$ in $W$ and $(-\varepsilon_D,0] \times S^1$ are taken so small such that $X_W$ and $X_D$ are outward pointing on each level set $\{t\} \times \partial W$ and $\{r\} \times S^1$, respectively.
We will parametrize the rounded corner explicitly:
choose real-valued functions $f,g$ defined on $[0,1]$ with the following properties
\begin{itemize}
\item $f$ satisfies $f|_{[0,1/4]}\equiv 0$, and is strictly decreasing on $(1/4,1]$ with $f(1)=-\varepsilon_W$.
\item $g$ is a strictly increasing function on $[0,3/4)$ with $g(0)=-\varepsilon_D$ and $g|_{[3/4,1]}\equiv 0$. 
\end{itemize}
Now put
\[
\psi:\partial W \times (0,1) \times \partial D^2 \longrightarrow
W \times D^2,
\quad
(w,t,z) \longmapsto
(Fl^{X_W}_{f(t)}(w),
Fl^{X_D}_{g(t)}(z)
)
\]
The Liouville form $\lambda_W+\lambda_D$ pulls back under $\psi$ to
\begin{equation}
\label{eq:form_rounded_corner}
e^{f(t)} i_W^* \lambda_W
+
e^{g(t)} i_D^* \lambda_D.
\end{equation}
Away from the rounded corner, the manifold $\partial_+ (W\times D^2)$ consists of the disjoint union of 
$$
(W \setminus (-\varepsilon_W,0] \times \partial W)\times S^1
\text{ and }
\partial W \times D^2_{<1-\varepsilon_D}.
$$
On the former set, the Liouville form restricts to the contact form $\lambda_W+d\phi$: we identify this set together with its contact form with the trivial mapping torus $M(W,\id)$ which is endowed with the same contact form.
On the set $\partial W \times D^2$, the Liouville form restricts to $\lambda_W|_{\partial W}+r^2 d\phi$. Together with the rounded corner, parametrized by $\psi$, we identify this contact manifold with a neighborhood of the binding which is equipped with a contact form that has the same form as defined in Equation~\eqref{eq:contact_form_binding} (use Equation~\eqref{eq:form_rounded_corner} to see this).

We hence get an abstract contact open book on the convex boundary. The choice of rounding data only affects the profile functions in Equation~\eqref{eq:contact_form_binding} and for any two profile functions that are admissible, a convex combination of them also is.
From now on, we will implicitly round corners and just write
\begin{equation}
\label{eq:open_book_boundary}
\partial_+(W\times D^2) \cong \OB(W,\id).
\end{equation}
In the following, we will decompose a Liouville domain as $W=W_c \cup W_m$, where $W_m \cong (-\varepsilon,0] \times \partial W$. We note that Equation~\eqref{eq:open_book_boundary} holds for $W_m$ as well, in which case we get an open book on a non-compact contact manifold.
We now make a sequence of simple observations:
\begin{enumerate}
\item suppose we are given $S:S^{k-1}\to \partial W$, an isotropic embedding into $\partial W$, together with a trivialization $\vrepsilon$ of its conformal symplectic normal bundle.
The embedding $S$ extends to an embedding into $\partial W\times D^2$, which we continue to denote by $S$. We also obtain a trivialization of the new larger conformal symplectic normal bundle by using $\partial_x$ and $\partial_y$, the tangent vectors to $\partial W \times \{0\}$ in the disk direction $D^2$, to trivialize the new directions in $CSN_{\partial_+ (W\times D^2)}(S)$. We denote this new trivialization by $\vrepsilon \oplus \vrepsilon_D$.

By choosing an appropriate (large) attaching region for the handle, we obtain a symplectomorphism, namely the identity,
\begin{equation}
\label{eq:fat_attachment}
\widetilde{W\times D^2}|_{S\times \id, \vrepsilon\oplus \vrepsilon_D}
\cong
\widetilde{W}|_{S,\vrepsilon} \times D^2.
\end{equation}

\item by applying in succession Equation~\eqref{eq:open_book_boundary}, observation (1), relating handle attachment to surgery as in Equation~\eqref{eq:handle_attach_induces_surgery}, and Equation~\eqref{eq:open_book_boundary} again, we find
$$
\OB(\widetilde{W},\id) \cong \partial_+(\widetilde{W}\times D^2)\cong \partial_+ (\widetilde{W\times D^2}) \cong 
\overline{ \partial_+ (W\times D^2)}
\cong \overline{\OB(W,\id)}
.
$$
\end{enumerate}
We now find (see also \cite{vK}, Proposition 4.2)
\begin{corollary}
\label{cor:subcritical_surgery_OB}
Suppose that $\OB(W,\psi)$ is a contact open book.
Suppose that $S:S^{k-1}\to B=\partial W$ is an isotropic embedding into the binding of the open book, and assume that we are given a trivialization of its conformal symplectic normal bundle of the form $\vrepsilon_B\oplus \vrepsilon_D$.
Then
\begin{equation}
\label{eq:surgery_OB}
\overline{\OB(W,\psi)}|_{S,\vrepsilon_B\oplus \vrepsilon_D} \cong 
\OB( \widetilde{W}|_{S,\vrepsilon_B},\tilde \psi),
\end{equation}
where $\tilde \psi$ equals $\psi$ on the complement of the attached $k$-handle and is the identity on the $k$-handle.
\end{corollary}
\begin{proof}
Decompose the Liouville filling $W=W_c \cup W_m$, where $W_m$ is the ``margin'' of the page, i.e.~a neighborhood of the (convex) boundary, where the monodromy $\psi$ equals the identity. The Liouville domain $W_c$ is the content of the page, i.e.~the complement of the margin.
We have
$$
\OB(W,\psi)
=M(W_c,\psi) \cup M(W_m,\id) \cup B(W) 
=M(W_c,\psi) \cup \OB(W_m,\id)
$$  
We will now apply observation (2) to $\OB(W_m,\id)$.
As the surgery takes place in a neighborhood of the binding we have
$$
\overline{\OB(W,\psi)} =
M(W_c,\psi) \cup \overline{\OB(W_m,\id)}
\cong 
M(W_c,\psi) \cup
\OB(\widetilde{W_m},\id)
=\OB(\widetilde{W},\tilde \psi).
$$
This proves the claim.
\end{proof}

As a special case we obtain the following.
\begin{corollary}
There is a contactomorphism between the following contact manifolds, each presented as abstract open books,
\begin{equation}
\label{eq:open_books_are_contacto}
\OB(W_1,\psi_1)\# \OB(W_2,\psi_2) \cong
\OB(W_1 \natural W_2,\psi_1 \natural \psi_2).
\end{equation}
\end{corollary}
Indeed, for this, we take $k=1$ in Corollary~\ref{cor:subcritical_surgery_OB}, put $W=W_1\coprod W_2$, and choose $S:S^0 \to \partial (W_1\coprod W_2)$ mapping one point of $S^0$ to one boundary component of $W_1$ and the other point of $S^0$ to a boundary component of $W_2$; there is only one homotopy class of trivializations of the conformal symplectic normal bundle in this case.
The left-hand side in \eqref{eq:surgery_OB} then reduces to $0$-surgery along $S$, which is the contact connected sum of the open books.
The right-hand side is a contact open book whose page is obtained by attaching a symplectic $1$-handle to $W=W_1\coprod W_2$: this is the new page $\tilde W$. The monodromy extends as the identity on the symplectic $1$-handle. In other words, the right-hand side is the abstract book-connected sum.

\end{document}